\documentclass[preprint]{elsarticle}

\usepackage{tikz}
\usepackage{pgfplots}
\usepackage{amsmath}
\usepackage{amsthm}
\usepackage{amssymb}
\usepackage{amsfonts}
\usepackage{enumerate}
\usepackage{bm}
\usepackage{ulem}
\usepackage{mathtools}
\usepackage{mathrsfs}
\usepackage{siunitx}
\usepackage{makecell}
\usepackage{tikz-cd}
\usepackage{graphicx}
\usepackage{thmtools}
\usepackage{thm-restate}
\usepackage{hyperref}
\hypersetup{colorlinks, citecolor=Blue, urlcolor=cyan}
\usepackage{xcolor}
\usepackage{wrapfig}
\usepackage[algo2e,boxed,vlined,linesnumbered,
titlenumbered]{algorithm2e}
\pgfplotsset{yticklabel style={
        /pgf/number format/fixed,
        /pgf/number format/precision=5
},
scaled y ticks=false,
compat=1.17}

\newcommand\N{\mathbb{N}}

\newcommand\K{\mathbb{K}}

\newcommand\V{\mathbf{V}}

\newcommand\SFGLM{\textsf{Sparse-FGLM}\xspace}
\newcommand\FGLM{\textsf{FGLM}\xspace}
\newcommand\gb{Gr\"obner basis\xspace}

\newcommand\bx{x_1, \dots, x_n}

\DeclarePairedDelimiter{\floor}{\lfloor}{\rfloor}
\DeclarePairedDelimiter{\ideal}{\langle}{\rangle}

\newcommand\Fcinq{\textsc{\texorpdfstring{F\textsubscript{5}}{F5}}\xspace}

\newtheorem{theorem}{Theorem}
\newtheorem{lemma}[theorem]{Lemma}
\newtheorem{corollary}[theorem]{Corollary}

\newtheorem{proposition}[theorem]{Proposition}
\newtheorem{definition}[theorem]{Definition}

\newtheorem{remark}[theorem]{Remark}

\usepackage[textwidth=3cm,backgroundcolor=yellow]{todonotes}

\def\ab#1{\noindent{\color{black}#1}}

\def\ms#1{\textcolor{black}{#1}}

\def\rev#1{\noindent{\color{black}#1}}

\begin{document}

\begin{frontmatter}

\title{Gr\"obner bases and critical values: \\The asymptotic combinatorics of determinantal systems}
\author[1]{J\'er\'emy Berthomieu}
\author[2]{Alin Bostan}
\author[1]{Andrew Ferguson}
\author[1]{Mohab Safey El Din}
\address[1]{Sorbonne Universit\'e, CNRS, LIP6, F-75005, Paris, France} 
\address[2]{Inria, Université Paris-Saclay, 1 rue Honoré d’Estienne d’Orves, 91120 Palaiseau, France}

\begin{abstract}
    Determinantal polynomial systems are those involving maximal minors of some given matrix. An important situation where these arise is the computation of the critical values of a polynomial map restricted to an algebraic set. This leads directly to a strategy for, among other problems, polynomial optimisation. 
    
    Computing Gr\"obner bases is a classical method for solving polynomial systems in general. 
    For practical computations, this consists of two main stages. First, a Gr\"obner basis is computed
    with respect to a DRL (degree reverse lexicographic) ordering. Then, a change of ordering algorithm, such as \textsf{Sparse-FGLM}, designed by Faug\`ere and Mou, is used to find a Gr\"obner basis of the same system but with respect to a lexicographic ordering. The complexity of this latter step, in terms of the number of arithmetic operations in the ground field, is $O(mD^2)$, where $D$ is the degree of the ideal generated by the input and $m$ is the number of non-trivial columns of a certain $D \times D$ matrix.
    
    While asymptotic estimates are known for $m$ in the case of \textit{generic} polynomial systems, 
    thus far, the complexity of  \textsf{Sparse-FGLM} was unknown for the class of determinantal systems. 
    
    By assuming Fr\"oberg's conjecture, thus ensuring that the Hilbert series of generic determinantal ideals have the necessary structure, we expand the work of Moreno-Soc\'ias by detailing the structure of the DRL staircase in the determinantal setting. Then we study the asymptotics of the quantity $m$ by relating it to the coefficients of these Hilbert series.
    Consequently, we arrive at a new bound on the complexity of the \textsf{Sparse-FGLM} algorithm for generic determinantal systems and, in particular, for generic critical point systems.
    
    We consider the ideal inside the polynomial ring $\mathbb{K}[x_1, \dots, x_n]$, where $\mathbb{K}$ is some infinite field, generated by $p$ generic polynomials of degree $d$ and the maximal minors of a $p \times (n-1)$ polynomial matrix with generic entries of degree $d-1$. Then, in this setting, for the case $d=2$ and for $n \gg p$ we establish an exact formula for $m$ in terms of $n$ and $p$. Moreover, for $d \geq 3$, we give a tight asymptotic formula, as $n \to \infty$, for $m$ in terms of $n,p$ and $d$.

       
\end{abstract}

\end{frontmatter}

\section{Introduction}
\paragraph{Motivation}
By the Lagrange multiplier theorem, the local extrema of a polynomial mapping restricted to a real algebraic set are \ms{contained in} the \ms{set of} critical values of the map. Thus, computing these values, and the corresponding minimum/critical points where these extrema are reached, leads to a strategy for polynomial optimisation under some regularity assumptions. 

Polynomial optimisation is of principal importance in many areas of engineering and social sciences (including control theory~\cite{henrion2005,henrion2003opt}, computer vision~\cite{Aholt_2013,probst2019convex} and optimal design~\cite{castro2017approximate}, etc.). 

Critical point computations are also a fundamental task in the algorithms of effective real algebraic geometry. For example, the problems of deciding the emptiness of the set of real solutions of a polynomial system, counting the number of connected components of such sets and one block quantifier elimination can all be accomplished, under some regularity assumptions, by the so-called critical point method~\ab{\cite[Ch.~7]{BaPoRo06}, see also} \cite{hong2012variant,safeyschost2003}.

With $\K$ an infinite field, let $f = (f_1, \dots, f_p) \in \K[x_1, \dots, x_n]$ be a sequence of polynomials of degree $d$ and let $\V(f) \subset \K^n$ be their simultaneous vanishing set. Define $\varphi_1$ to be the projection map onto the first coordinate. We denote by $\mathcal{J}$ the Jacobian of $(\varphi_1, f)$,
   \[ \mathcal{J} :=
      \begin{bmatrix}
        1 & 0 & \cdots  & 0 \\ 
        \frac{\partial f_1}{\partial x_1} & \frac{\partial f_1}{\partial x_2} &  \cdots & \frac{\partial f_1}{\partial x_n} \\
        \vdots & \vdots & \ddots & \vdots\\
        \frac{\partial f_p}{\partial x_1} & \frac{\partial \ab{f_p}}{\partial x_2} &  \cdots & \frac{\partial f_p}{\partial x_n}
      \end{bmatrix}.
    \]      
An example of the ideals we consider in this paper is the ideal $I$ defined by~$f$ and the maximal minors of $\mathcal{J}$. By a corollary of the Jacobian criterion~\cite[Corollary 16.20]{eisenbud}, when $f$ is a reduced regular sequence and $\V(f)$ is smooth, the algebraic set $\V(I)$ is exactly the set of critical points of the projection map $\varphi_{1}$ restricted to the algebraic set $\V(f)$. 

Throughout this paper, we shall consider \textit{generic} determinantal systems. Essentially, for the example of critical point systems, we choose the coefficients of the polynomials $f_1, \dots, f_p$ so that they lie inside a non-empty Zariski open subset of $\K^{\binom{n+d}{d}}$ where the results of~\cite{unmixed} hold. In particular, the generic systems we consider satisfy the conditions of the Jacobian criterion so that $I$ encodes the critical points of $\varphi_{1}$ restricted to $\V(f)$~\cite[Lemma A.2]{appendixsafey}. Moreover, by~\cite[Lemma 2]{unmixed} and~\cite[Proposition 4.2]{labahn2020homotopy}, $I$ is a zero-dimensional, radical ideal. So, the quotient algebra $\K[x_1, \dots, x_n]/I$ is a finite dimensional vector space \ab{over $\K$}.

For the many applications of the critical point method \ab{previously discussed}, one wishes to compute a rational parametrisation of this set of critical points. By our genericity conditions, we shall assume that the ideal $I$ is in \emph{shape position}, meaning that for a lexicographic (LEX) ordering with $x_n$ as the least variable, the LEX \gb has the following structure:
\[
    \{ x_1 - g_1(x_n), \dots, x_{n-1} - g_{n-1}(x_{\ab{n}}), g_n(x_n) \},
\]
where the degree of $g_n$ is the degree of the ideal $I$~\cite{shapelemma}. A fast method commonly used in practice, and the one which we shall use, to compute a LEX \gb is to first compute a \gb of $I$ with respect to a degree reverse lexicographic ordering (DRL). Then, one uses a change of ordering algorithm to compute another \gb of $I$ but with respect to a LEX ordering.

\paragraph{Previous works}
In \cite[Theorem 3]{unmixed}, \ab{Faug\`ere, Safey El Din and Spaenlehauer} give an upper bound on the number of arithmetic operations necessary for computing  
a LEX \gb of a generic determinantal system within the DRL to LEX framework. 
\ab{They do so by deriving the Hilbert series of such a system, using results by Conca and Herzog}~\cite[Corollary~1]{conca1994hilbert}.

Then, \ab{based on a result of Bardet, Faugère and Salvy}~\cite[Theorem~7]{bardet2004}, the authors \ab{of~\cite[Theorem~3]{unmixed}} analyse the complexity of the DRL step using Faug\`ere's \Fcinq algorithm~\cite{f5}. 
\ab{Here, and in the whole text, complexity estimates are given in terms of arithmetic operations in the ground field~$\K$.}
Next, to obtain a LEX \gb, since we are in the zero-dimensional case, they use the \FGLM algorithm to perform the change of ordering~\cite{fglm}.
The complexity of \FGLM is $O(nD^3)$, where $D$ is the degree of the determinantal ideal. In~\cite[Theorem 2.2]{nie2008algebraic}, \ab{ Nie and Ranestad} use the Thom-Porteous-Giambelli formula to prove that this degree is \[ D = d^p (d-1)^{n-p} \binom{n-1}{p-1}.\]

In~\cite{sparsefglm}, \ab{Faugère and Mou proposed} another algorithm that solves the change of ordering step, the \SFGLM algorithm. Under some genericity assumptions, \SFGLM relies primarily on the structure of the matrix $M_n$ associated to \ab{the linear map of} multiplication by $x_n$ in the finite dimensional quotient algebra $\K[x_1, \dots, x_n]/I$. Its complexity is $O(mD^2 + nD\log^2D)$, where $m$ is the number of non-trivial columns of the matrix $M_n$. This number is studied in the same paper for generic complete intersections using the results of Moreno-Soc\'ias~\cite{moreno}. By deriving the asymptotics of the number of non-trivial columns, as well as by proving that the structure of the matrix~$M_n$ is such that it can be computed free of arithmetic operations, \ab{Faugère and Mou} demonstrate \ab{in~\cite{sparsefglm}} that the complexity of \SFGLM is indeed an improvement of that of \FGLM. 

\paragraph{Main results}
In this paper, under similar genericity assumptions and by assuming a variant of Fr\"oberg's conjecture~\cite{froberg}, we extend the results of~\cite{sparsefglm, moreno} to generic determinantal ideals. We emphasise here that our results hold not only for critical point systems but indeed for any sufficiently generic determinantal system. This is made precise in Definition~\ref{def:gdi}.

Firstly, we prove a result on the structure of the DRL staircase, which implies that the only non-trivial columns of $M_n$ correspond one-to-one with 
monomials which, once multiplied by $x_n$, give a leading monomial in the reduced DRL \gb. Furthermore, for each such monomial, one can read the entries of \ab{the corresponding} non-trivial column from the polynomial in the \gb with that leading monomial. This implies the following theorem.
    
\begin{restatable}{theorem}{Mn}\label{thm:mn}
Let $I$ be a generic determinantal ideal so that the conditions of Definition~\ref{def:gdi} hold. Assume \ab{that} a reduced and minimal \gb of $I$ with respect to a DRL ordering is known. Then the multiplication matrix $M_n$ can be constructed without performing any arithmetic operations.
\end{restatable}

Continuing further, we prove an explicit formula for the number of non-trivial columns of $M_n$, which we denote $m$, in the case of quadratic polynomials with a large number of variables~$n$ compared to the number of polynomials~$p$. Then, for any choice of degree $d \geq 3$ and for $n \to \infty$, we prove asymptotic formulae for $m$.

\begin{restatable}{theorem}{m}\label{thm:m}
Let $I$ be a generic determinantal ideal so that the conditions of Definition~\ref{def:gdi} hold, and let $M_n$ be the matrix associated to the linear map of multiplication by $x_n$. Denote by $m$ the number of non-trivial columns of $M_n$. 
Then, for $d = 2$ and $n \gg p$, 
\begin{equation}\label{exact:d=2} 
m = \sum_{k=0}^{p-1} \binom{n-p-1+k}{k}\binom{p}{\floor{3p/2}-1-j}.
\end{equation}
Moreover, for $d\geq3$ and $n \to \infty$, 
\begin{equation}\label{asympt:d>2} 
m \approx \frac{1}{\sqrt{(n-p)\pi}}\sqrt { \frac{6}{(d-1)^2-1} } d^p (d-1)^{n-p} \binom{n-2}{p-1}.
\end{equation}

\end{restatable}

By~\cite[Theorem~3.2]{sparsefglm}, and since the ideals we consider are in shape position, Theorem~\ref{thm:m} leads directly to a complexity result for the \SFGLM algorithm. Therefore, we arrive at an improved upper bound on the complexity of the change of ordering step for generic determinantal systems.

\begin{restatable}{theorem}{comp}\label{thm:comp}
Let $I$ be a generic determinantal ideal so that the conditions of Definition~\ref{def:gdi} hold. Assume \ab{that} a reduced and minimal DRL \gb of $I$ is known. Then, for $d \geq 3$, the arithmetic complexity of computing a LEX \gb of~$I$ is upper bounded by 
\[
    O\left( \frac{d^{3p} (d-1)^{3(n-p)}}{\sqrt{(n-p)d\pi}}  \binom{n-2}{p-1} \binom{n-1}{p-1}^2 \right).
\]
Hence, the complexity gain of \SFGLM over \FGLM for generic determinantal systems is approximately
\[
    O\left( \frac{m}{nD}\right) \approx O\left( \frac{\sqrt{n-p}}{n^2(d-1)}\right) .
\]
\end{restatable}

\paragraph{Organisation of the paper}
The remainder of the paper consists of: Section~\ref{sec:pre}, where we define the class of ideals for which our results hold; Section~\ref{sec:proofs}, where we prove our main results; and Section~\ref{sec:exp}, where we test our formula for the number of non-trivial columns of the matrix $M_n$ for various parameters. 

\section{Preliminaries}
\label{sec:pre}
\subsection{Shape position}
Let $f_1, \dots, f_p \in  \K[x_1, \dots, x_n]$ be polynomials of degree $d$. Similarly, let $h_{1,2}, \dots, h_{p,n} \in \K[x_1, \dots, x_n]$ be polynomials of degree $d-1$. Let $I$ be the ideal generated by $\ideal{f_1, \dots, f_p}$ and the maximal minors of the following matrix:
  \[ 
      \begin{bmatrix}
        h_{1,2} & \cdots & h_{1,n} \\
        \vdots & \ddots & \\
        h_{p,2} & \cdots & h_{p,n}
      \end{bmatrix}.
    \] 
    
The authors of~\cite[Proposition 4.2]{labahn2020homotopy} show that if the coefficients of $f_1,\dots,f_p$ and $h_{1,2},\dots,h_{p,n}$ are chosen in some non-empty Zariski open subsets of $\K^{\binom{n+d}{d}}$ and $\K^{\binom{n+d-1}{d-1}}$ respectively, then the ideal $I$ defined above is radical and zero-dimensional.

In order to apply the results of~\cite{sparsefglm} to our determinantal ideals, we require they be in shape position. To ensure this, we add a new indeterminate \rev{that acts as a primitive element of the quotient algebra}. For any $\lambda \in \K^n$, define the ideal \[J = I + \ideal{y - \sum_{j=1}^n \lambda_j x_j} \quad \subset \quad \K[x_1, \dots, x_n, y].\] \rev{The idea of the following lemma is similar to that of applying a generic linear change of variables to the ideal $I$. However, introducing a new variable to be the least variable in the monomial ordering also allows one to avoid some degenerate cases that will be discussed in Remark~\ref{rem:shape}.}

\begin{lemma}\label{lem:shape}
    Let $\K$ be an infinite field. Then, there exists a non-empty Zariski open subset $\mathcal{O}$ of $\K^n$ such that for all $\lambda \in \mathcal{O}$ and with $y$ as the least variable in the LEX ordering, the ideal $J$ is in shape position.
\end{lemma}
    
\begin{proof}
    By~\cite[Proposition 4.2]{labahn2020homotopy}, the ideal $I$ is radical and zero-dimensional. Thus, for all $\lambda \in \K^n$, the ideal $J$ is also zero-dimensional and radical. By~\cite[Proposition 1.6]{gianni1987algebraic}, \cite[Proposition 5]{shapelemma} and the genericity of the polynomials defining $I$, we have that $J$ is in shape position if and only if each of the finitely many points in the algebraic set $\V(J)$ has a unique $y$-coordinate. As $\K$ is infinite, the finitely many linear equations that give equality of the $y$ coordinate of any two points in $\V(J)$ define a proper Zariski closed subset of $\K^n$. Therefore, there exists a non-empty Zariski open subset $\mathcal{O}$ of $\K^n$ such that for all $\lambda \in \mathcal{O}$ the $y$ coordinate of each point in the algebraic set $\V(J)$ is unique. Hence, for $\lambda \in \mathcal{O}$, the ideal $J$ is in shape position.
\end{proof}

\subsection{Fr\"oberg's conjecture}
As a direct consequence of~\cite{conca1994hilbert}, the authors of~\cite{unmixed} further show that under the same genericity assumptions, the following proposition holds:   
\begin{proposition}[{\cite[Proposition 1]{unmixed}}]\label{prop:original_H}
   The Hilbert series of $\K[x_1, \dots, x_n] / I$ is
    \[
      H = \frac{\det(P(t^{d-1}))}{t^{(d-1)\binom{p-1}{2}}} \frac{(1-t^d)^p(1-t^{d-1})^{n-p}}{(1-t)^n}
    \]
    where $P(t)$ is the $(p-1) \times (p-1)$ matrix whose $(i,j)th$ entry \ab{is} $\sum_k \binom{p-i}{k} \binom{n-1-j}{k} t^k$.
\end{proposition}
We shall consider the quotients of the algebra $\K[x_1,\ldots,x_n]/I$ by powers of generic linear forms. By the genericity introduced above, it suffices to consider the quotients of $A = \K[x_1, \dots, x_n, y] / J$ by powers of~$y$. Thus, denote by $HQ_e$ the Hilbert series of $A/\ideal{y^e}$, for $e \geq 1$. In order to control the shape of this Hilbert series, we rely on a variant of Fr\"oberg's conjecture given in~\cite[Lemma~14]{genminrank}. First, however, a definition.

\begin{definition}\label{def:quo}
  For a series $S = \sum_k a_k t^k$, we define
  \[
    \left[\sum_k a_k t^k \right]_+  
  \]
  to be the series $S$ truncated at the first non-positive coefficient.
\end{definition}
\begin{lemma}[{\cite[Lemma 14]{genminrank}}]\label{lem:froberg}
  If Fr\"oberg's conjecture is true, then for all $e \geq 1$
\[
   HQ_e =  \left[ (1-t^e)H \right]_+.
\]
\end{lemma}
We remark that in~\cite{Pardue10}, Pardue showed that Moreno-Soc\'{\i}as' conjecture~\cite[Conjecture 4.2]{moreno} implies Fr\"oberg's conjecture, as well as a number of other interesting conjectures. Moreover, while these conjectures are usually given in a homogeneous setting, we shall assume that Lemma~\ref{lem:froberg} holds also in the affine case.
\subsection{Generic determinantal ideals}
With the assumption of Fr\"oberg's conjecture, we define precisely the class of ideals we consider in this paper.

\begin{definition}\label{def:gdi}
Let $I \subset \K[x_1,\dots,x_n]$ be an ideal with $\K$ an infinite field. We say that $I$ is a generic determinantal ideal if the following three conditions hold:
\begin{itemize}
    \item the ideal $I$ is in shape position,
    \item the Hilbert series \rev{$H$} of $\K[x_1,\dots,x_n]/I$ is given by Proposition~\ref{prop:original_H},
    \item \rev{as in Lemma~\ref{lem:froberg}, the Hilbert series $HQ_e$ of $\left(\K[x_1,\dots,x_n]/I\right)/\ideal{x_n^e}$ is equal to $\left[ (1-t^e)H \right]_+$ for all $e \geq 1$.}   
\end{itemize}
\end{definition}
By our genericity assumptions, $f_1, \dots, f_p$ is a reduced, regular sequence defining a smooth algebraic set $\V(f_1, \dots, f_p)$. By~\cite[Lemma~6]{unmixed} and~\cite[Ch.~9, Sec.~3, Prop.~9]{CLO}, the determinantal ideal defining the critical points
of the projection map onto the first coordinate restricted to $\V(f_1, \dots, f_p)$ satisfies Proposition~\ref{prop:original_H}. Moreover, using the same addition of a new indeterminate \ab{as} in Lemma~\ref{lem:shape}, one may assume that such an ideal is \ab{in} shape position. Thus, by assuming Fr\"oberg's conjecture, these generic critical point systems are an important example of the generic determinantal ideals we consider.

\begin{remark}\label{rem:shape}
We note that without the addition of a new indeterminate, generic critical point systems may not satisfy the conditions of Definition~\ref{def:gdi}. In particular, if one considers a DRL ordering with $x_n$ as the least variable for the determinantal system defining the critical values of the projection map onto the $x_n$-axis, then Lemma~\ref{lem:froberg} no longer holds. The consequence of this is that the results of this paper cannot then be applied to this special case. However, introducing a new indeterminate to be the least variable in the DRL ordering, as in Lemma~\ref{lem:shape}, rectifies this problem. Therefore, we may assume that all generic critical point systems satisfy the conditions of Definition~\ref{def:gdi}.
\end{remark}

Furthermore, note that the Hilbert series of $\K[x_1, \dots, x_n] / I$ is equal to the Hilbert series of $\K[x_1, \dots, x_n, y] / J$. Therefore, for ease of notation, we shall assume that the determinantal ideals considered in this paper satisfy Definition~\ref{def:gdi} without introducing the new indeterminate $y$.

\section{Proofs}\label{sec:proofs}

\paragraph{Roadmap} Firstly, as in the papers~\cite{sparsefglm, unmixed,moreno}, to prove our results we rely on manipulations of the Hilbert series $H$ \ms{from Proposition~\ref{prop:original_H}}. However, for our purposes, the form involving the determinant of the matrix $P$ makes this difficult. Thus, our first step is to express $H$ in a simpler form in Section~\ref{sec:det}. Then we show that this Hilbert series is always \ab{\emph{unimodal}} in Section~\ref{sec:unimodal}. This property, along with the assumption of Fr\"oberg's conjecture, allows us to prove in Section~\ref{sec:structure} a structure theorem on the generic DRL staircase. This leads to our first main result, that the multiplication matrix $M_n$ can be constructed for free. Combining this result with the \ab{unimodality} property, we show that the number of non-trivial columns of this matrix, a key parameter of the \SFGLM algorithm, is equal \ab{to} the largest coefficient of the series~$H$. In Section~\ref{sec:asymp}, we conclude the proof of our main results by studying the asymptotics of the largest coefficient of~$H$. 

\subsection{Simplification of the Hilbert series}\label{sec:det}
As in the works we wish to generalise~\cite{sparsefglm, unmixed, moreno}, our results rely heavily on the Hilbert series of the generic determinantal ideals we consider. Thus, the first stage we take is to simplify the form given in Proposition~\ref{prop:original_H}. We do so by expressing the determinant of the binomial matrix in this Hilbert series as a binomial sum. We start with some general results involving binomial matrices that will lead to the simplification we want as a special case.

Let $\mathcal{A} = (a_{i j})_{i,j \geq 0}$ be the infinite Pascal matrix defined by $a_{ij} = \binom{i}{j}$ for $j \leq i$ and $a_{i j} = 0$ for $j > i$. The minor of this matrix corresponding to rows $0 \leq a_1 < \cdots < a_n$ and columns $0 \leq b_1 < \cdots < b_n$
will be denoted by \[ \binom{a_1, \dots, a_n}{b_1, \dots, b_n} = 
\begin{vmatrix}
\binom{a_1}{b_1} & \cdots & \binom{a_1}{b_n} \\ 
\vdots & \ddots &  \vdots \\ 
\binom{a_n}{b_1} & \cdots & \binom{a_n}{b_n}   
\end{vmatrix} .
\]

We recall the following two lemmas from~\cite{gessel}.
\begin{lemma}[{\cite[Lemma 8]{gessel}}]\label{lem:binomreduction}
    If $b_1 \neq 0$, then
    \[
  \binom{a_1, \dots, a_k}{b_1, \dots, b_k}  = \frac{a_1 \cdots a_k}{b_1 \cdots b_k}  \binom{a_1-1, \dots, a_k-1}{b_1-1, \dots, b_k-1} .
    \]
\end{lemma}

\begin{lemma}[{\cite[Lemma 9]{gessel}}]\label{lem:rowreduction}
The following holds
    \[
         \binom{a, a+1, \dots, a+k-1}{0, b_2, \dots, b_k}  =  \binom{a, a+1, \dots, a+k-2}{b_2-1, b_3-1 ,\dots, b_k-1} .
    \]
\end{lemma}
We can now prove the following identity.
\begin{lemma}\label{lem:submatrix}
    Let $S$ be the $k \times (k+1)$ submatrix corresponding to rows $a+1, a+2, \dots, a+k$ and columns $0, 1, \dots, k$. Then,
    for $0 \leq l \leq k$, the minors of this submatrix are equal to
    \[   \binom{a+1, a+2, \dots, a+k}{0, 1, \dots \ell-1, \ell+1, \dots, k}   = \binom{a+k-\ell}{k-\ell} . 
    \]
\end{lemma}
\begin{proof}
Apply Lemma~\ref{lem:rowreduction} $\ell$ times to the minor \[ \binom{a+1, a+2, \dots, a+k}{0, 1, \dots \ell-1, \ell+1, \dots, k} . \] 
The result is the minor \[ \binom{a+1, a+2, \dots, a+k-\ell}{1, \dots, k-\ab{\ell}}. \] Next, apply Lemma~\ref{lem:binomreduction} to
obtain the minor \[ \frac{(a+1) \cdots (a +k-\ell)}{1 \cdots (k-\ell)}  \binom{a,\dots, a-1+k-\ell}{0, 1, \dots, k-\ell-1} = \binom{a+k-\ell}{k-\ell} \binom{a, \dots, a-1+k-\ell}{0, 1, \dots, k-\ell-1}. \] Finally, apply Lemma~\ref{lem:rowreduction} another $k-\ell-\rev{1}$ times until the minor is reduced to a single entry
\[ \binom{a+k-\ell}{k-\ell} \binom{a}{0} = \binom{a+k-\ell}{k-\ell}. \qedhere \] 
\end{proof}

\begin{lemma}\label{lem:cauchybinet}
    Let $M$ be the $m \times m$ matrix with entries in $\K[x,y,t]$ defined by \\
    $M_{i,j} = \sum_{\rev{k = 0}}^{\rev{m}} \binom{x-i}{k} \binom{y-j}{k} t^k$. 
    Then
    \[ \frac{\det(M)}{t^{\binom{m}{2}}} = \sum_{k=0}^{m} \binom{\rev{x-m-1+k}}{k}\binom{\rev{y-m-1+k}}{k} t^k . \]
\end{lemma}
\begin{proof}
Let $A$ be the $m \times (m+1)$ matrix with entries $a_{ik} = \binom{x-i}{k-1}$. Let $B$ be the $(m+1) \times m$ matrix with entries
$B_{kj} = \binom{y-j}{k-1} t^{k-1}$. 
\ms{Observe that $M = AB$}.

We shall write $A^{[\ell]}$ (resp.\ $B^{[\ell]}$) for the matrix $A$ (resp.\ $B$) with its $\ell$th column (resp.\ row) removed.
\ab{By} the Cauchy-Binet formula
\[ \det(M) = \sum_{\ell=1}^{m+1} \det(A^{[\ell]})\det(B^{[\ell]}). \] 

We begin with the matrix $A$. Notice that by making $\frac{1}{2}m(m+1)$ column transpositions, one can rearrange $A$ so that it is a submatrix of the Pascal matrix $\mathcal{A}$. Specifically, one can rearrange the columns of $A$ so that it has rows $\rev{x-m, \dots, x-1}$ and columns $0, \dots, m$ of $\mathcal{A}$. Then, by Lemma~\ref{lem:submatrix}, the determinant of the minors of $A$ equals, up to the sign difference from the transpositions, \[ \det(A^{[\ell]}) = \pm \binom{\rev{x-\ell}}{m-\ell+1}. \]

Now, let $C$ be the matrix $B$ with $t=1$. Then note that \[ \det(B^{[\ell]}) = \det(C^{[\ell]}) t^{\binom{m}{2} + m - \ell + 1}.\] 

In the same way as for the matrix $A$, by taking the transpose of $C$ and making $\frac{1}{2}m(m+1)$ column transpositions, one can rearrange $C$ so that is has the form of a submatrix of~$\mathcal{A}$. We find that \[ \det(C^{[\ell]}) = \pm \binom{\rev{y-\ell}}{m-\ell+1}, \quad \text{and thus} \;  \det(B^{[\ell]}) = \pm \binom{\rev{y-\ell}}{m-\ell+1} t^{\binom{m}{2} + m - \ell + 1}.\]

Returning to the Cauchy-Binet formula,
\[ \det(M) = \sum_{\ell=1}^{m+1} \binom{\rev{x-\ell}}{m-\ell+1} \binom{\rev{y-\ell}}{m-\ell+1} t^{\binom{m}{2} + m - \ell + 1}. \] 

By a change of coordinates,
substituting $k = m-\ell+1$, we arrive at
\[ \det(M) = \sum_{k=0}^{m} \binom{\rev{x-m-1+k}}{k} \binom{\rev{y-m-1+k}}{k} t^{\binom{m}{2}+k} . \qedhere\] 
\end{proof}

\begin{corollary}\label{cor:new_H}
    The Hilbert series $H$ \ms{from Proposition~\ref{prop:original_H}} can be expressed as \[
    H = \left(\sum_{k=0}^{p-1} \binom{n-p-1+k}{k} t^{k(d-1)}\right) \frac{(1-t^d)^p(1-t^{d-1})^{n-p}}{(1-t)^n} .\]
\end{corollary}

\begin{proof}
By Proposition~\ref{prop:original_H}, 
    \[
      H = \frac{\det(P(t^{d-1}))}{t^{(d-1)\binom{p-1}{2}}} \frac{(1-t^d)^p(1-t^{d-1})^{n-p}}{(1-t)^n}
    \]
where $P(t)$ is the $(p-1) \times (p-1)$ matrix whose $(i,j$)th entry is $\sum_k \binom{p-i}{k} \binom{n-1-j}{k} t^k$. Thus,
as the special case of Lemma~\ref{lem:cauchybinet} with $m = \rev{p-1,} x = p$ and $y=n-1$,
 \[
   \frac{\det(P(t))}{t^{\binom{p-1}{2}}} = \sum_{k=0}^{p-1} \binom{n-p-1+k}{k} t^{k}. \qedhere
 \]
\end{proof}

\subsection{Unimodality}\label{sec:unimodal}
The Hilbert series of the systems we study are highly structured. In particular, it was shown in~\cite[Proposition 2.2]{moreno} that the Hilbert series of generic complete intersections are symmetric and so-called unimodal polynomials. As we transition to more general determinantal ideals, we may lose some of this structure for certain choices of parameters. However, we show in this section that our series are always unimodal. This property will then be exploited in the remaining two parts of Section~\ref{sec:proofs}. We begin with the definition of unimodality.
\begin{definition}\label{def:uni}
  A polynomial $\sum_{k=0}^n a_k t^k$ with non-negative coefficients is unimodal if there exists an integer $N$ such that
  \begin{align*}
    a_k \leq a_{k+1}\leq a_N\quad \text{for} \;\; k < N, 
    \quad &\text{and} \quad
    a_N\geq a_k \geq a_{k+1}\quad \text{for} \;\; k \geq N  .
  \end{align*}
\end{definition}

Unimodality is not necessarily preserved by multiplication. \ab{For example,  the polynomial $f = 3+t+t^2$
is unimodal, while $f^2 = 9+6t+7t^2+2t^3+t^4$ is not.}
\begin{definition}\label{def:strong_uni}
  A polynomial $f$ with non-negative coefficients is strongly unimodal if, for all unimodal polynomials $g$, the product $fg$ is
  unimodal.
\end{definition}

Note that a strongly unimodal polynomial is also unimodal. A classical example of a strongly unimodal \ab{polynomial} is as follows.

\begin{lemma}\label{lem:strong_unimodal_ex}
  \ab{For any $d\in\N$, the polynomial $f = 1+t+\dots+t^d$  is strongly unimodal.}
\end{lemma}

\begin{proof}
  Let $g = \sum_{k=0}^n a_k t^k$ be a unimodal polynomial with integer $N$ such that
  \begin{align*}
    a_k \leq a_{k+1} \leq a_N & \; \; \text{for all } k < N, \\
    a_N\geq a_k \geq a_{k+1} & \; \; \text{for all } k \geq N.
  \end{align*}
  For ease of notation, let $a_k = 0$ if $k < 0$ or $k > n$. 
  Let $fg = \sum_{k=0}^{n+d} b_k t^k$ so that $b_k = a_{k-d} + \dots + a_k$. Suppose that there does
  not exist an integer $\sigma$ such that $b_{\sigma+1} < b_\sigma$, then $fg$ is trivially unimodal.
  On the other hand, suppose such an index exists and let $M$ be the least integer such that
  $b_{M+1} < b_M$. Clearly,
  $M \geq N$, since the coefficients of $g$ are non-decreasing up to index $N$. Assume that for some $k$, for all $\ell$ such that $M \leq \ell < k$ we have that $b_{\ell+1} \leq b_\ell$. Then $a_k - a_{k-d-1} \leq 0$. Since $k+1 \geq M+1 > N$, by the unimodality of $g$, $a_{k+1} \leq a_k$. Similarly, if $k-d \leq N$ we have $a_{k-d-1} \leq a_{k-d}$. Hence, by the inductive assumption, $b_{k+1} - b_k = a_{k+1} - a_{k-d} \leq a_k - a_{k-d-1} \leq 0$.
  Alternatively, if $k-d > N$, then by unimodality of $g$ we have $a_{k+1} - a_{k-d} \leq 0$. Hence, by induction, $b_{k+1} \leq b_{k}$ for all $k > M$. Thus, $fg$ is a unimodal polynomial and we conclude that $f$ is a strongly unimodal polynomial.
\end{proof}

Unlike unimodality, strong unimodality is preserved by multiplication. 

\begin{lemma}\label{lem:strong_unimodal_multiply}
  Let $f,g$ be strongly unimodal polynomials. Then, $fg$ is a strongly unimodal polynomial.
\end{lemma}

\begin{proof}
  Let $h$ be a unimodal polynomial. Then, since $g$ is strongly unimodal, $gh$ is a unimodal polynomial.
  Hence, since $f$ is strongly unimodal, $fgh$ is unimodal and so $fg$ is strongly unimodal.
\end{proof}

We shall prove that the Hilbert series of a generic determinantal ideal is unimodal by showing that it
is the product of a strongly unimodal polynomial and a unimodal polynomial.

\begin{lemma}\label{lem:unimodal}
  Let $H$ be the Hilbert series from Proposition~\ref{prop:original_H}, with parameters $n,p,d \in \N$ where $n > p$. Then $H$ is a unimodal polynomial.
\end{lemma}

\begin{proof}
  Firstly, by Corollary~\ref{cor:new_H},
  \[
   H = \left(\sum_{k=0}^{p-1} \binom{n-p-1+k}{k} t^{k(d-1)}\right)\frac{(1-t^d)^p(1-t^{d-1})^{n-p}}{(1-t)^n} .
  \]
  Our strategy is to show that we can write this polynomial as the product of a unimodal polynomial and
  a strongly unimodal polynomial. The polynomial $H$ would then be unimodal by Definition~\ref{def:strong_uni}.
  
  For $d>2$, the binomial sum factor \[ \sum_{k=0}^{p-1} \binom{n-p-1+k}{k} t^{k(d-1)} \] is not unimodal. However, since $n \geq p-1$, the remaining factor of $H$ always has the following polynomial as a factor:
  \[
    \frac{1-t^{d-1}}{1-t} = 1 + t + \dots + t^{d-2}.
  \]
  Therefore, we can always multiply this factor into the binomial sum above
  \[
    \sum_{k=0}^{p-1} \binom{n-p-1+k}{k} t^{k(d-1)}(1 + t + \dots + t^{d-2}) =
    \sum_{k=0}^{p-1} \sum_{i=0}^{d-2} \binom{n-p-1+k}{k} t^{k(d-1) + i} .
  \]
  The resulting polynomial is unimodal as its coefficients are non-decreasing with no internal zeroes.

  Consider the remaining quotient
  \[
    \frac{(1-t^d)^p(1-t^{d-1})^{n-p-1}}{(1-t)^{n-1}}.
  \]
  This polynomial is the product of $n-1$ polynomials of the form $1 + t + \dots + t^m$ for some
  $m \in \N$. By Lemma~\ref{lem:strong_unimodal_ex}, each of these polynomials is strongly unimodal.
  Thus, by Lemma~\ref{lem:strong_unimodal_multiply}, the remaining quotient,
  \[
    \frac{(1-t^d)^p(1-t^{d-1})^{n-p-1}}{(1-t)^{n-1}} ,
  \]
  is strongly unimodal. Therefore, since $H$ is the product of a strongly unimodal polynomial and a unimodal polynomial, $H$ is unimodal.
\end{proof}

\begin{remark}
In the context of this paper, by the unimodality of the Hilbert series $H$, Definition~\ref{def:quo} is equivalent to the definition given in~\cite[Section~1]{moreno}.
\end{remark}
\subsection{Staircase structure}\label{sec:structure}

In this section, we prove a structure theorem on the DRL staircase for generic determinantal ideals. Let $(g_1,\dots,g_k)$ be a reduced and minimal \gb of $I$ with respect to a DRL ordering with $x_n$ as the least variable. For $1 \leq i \leq k$, let $r_i \in \mathscr{M}$ be the leading monomial of $g_i$, where $\mathscr{M}$ is set of monomials of $\K[x_1,\dots,x_n]$. Then we shall denote the DRL staircase by
\[
    E = \bigcap_{i=1}^k \big\{r \in \mathscr{M} \mid r_i \nmid r \big \}. 
\]
The elements of the staircase give a natural basis for the quotient algebra $\K[x_1,\dots,x_n]/I$. For each $b \in E$, the columns of the matrix $M_n$ are the normal forms of $x_n b$ with respect to the DRL \gb expressed in terms of the basis $E$. Thus, the construction of the column of~$M_n$ corresponding to $x_n b$ falls into exactly one of following three cases:
\begin{enumerate}
    \item $x_n b \in E$: Then the corresponding column is sparse, consisting of all zeroes except one entry with a value of $1$ in the row corresponding to $x_n b$.
    \item $x_n b$ is a leading term of the reduced DRL \gb: Then the normal form is obtained from the polynomial $g$ in the \gb whose leading term is $x_n b$.
    \item Otherwise, the normal form must be computed.
\end{enumerate}
In the first case, the corresponding column is trivial. In the latter two cases, the corresponding columns are non-trivial. Usually, and in the case we consider with generic polynomials, these non-trivial columns are dense. Moreover, constructing columns that fall into the first two cases do not require any arithmetic operations. 

We establish in this subsection that, for generic determinantal ideals, only the first two cases occur. This implies that the number of non-trivial columns of the matrix $M_n$ is equal to the number of leading monomials of elements of the reduced DRL \gb that have positive degree in $x_n$.

To prove this result, we consider the Hilbert series $H$, its simplified form from Corollary~\ref{cor:new_H} as well as the unimodal property of Lemma~\ref{lem:unimodal}. Here, we illustrate an example of the DRL staircase in the case $(d,p,n) = (3,2,3)$.
\begin{center}
\includegraphics[width=0.8\linewidth]{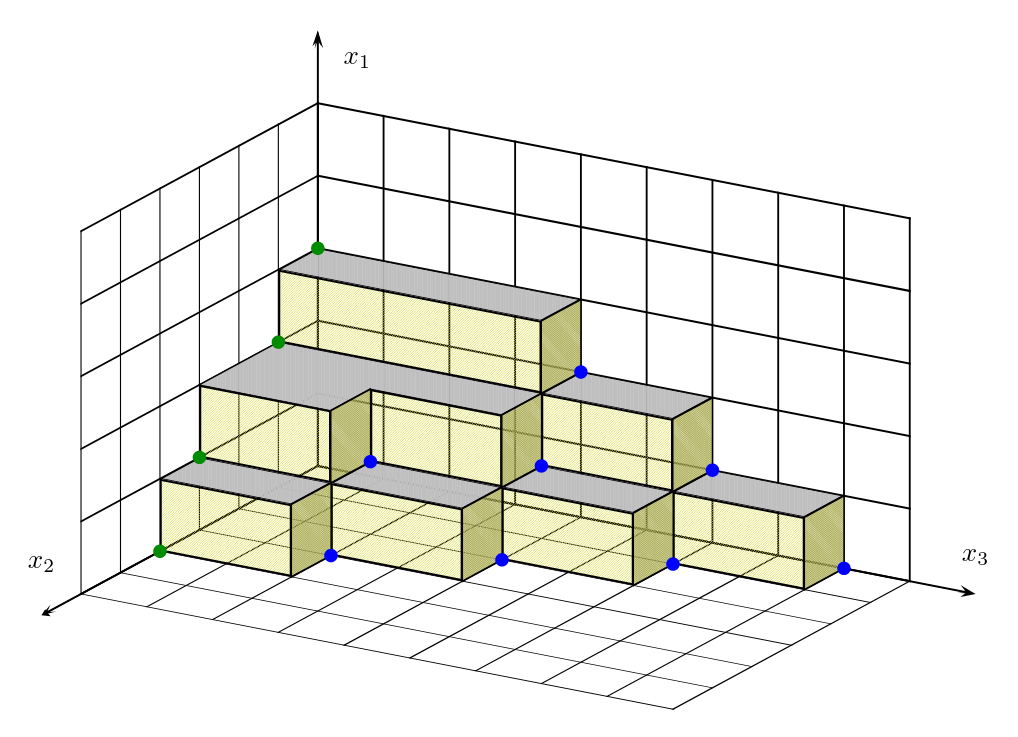}
\end{center}

Here, the cubes represent elements of the staircase and the dots are the leading monomials of the reduced DRL \gb. We can see that in this instance, the number of non-trivial columns is equal to the number of blue dots, the number of leading monomials of elements of the reduced \gb that have positive degree in $x_n$.

We recall the definition of $HQ_e$, the Hilbert series of $(\K[x_1,\ldots,x_n]/I)/\ideal{x_n^e}$, for $e \geq 1$. Also, recall that we assume that Fr\"oberg's conjecture is true and so the conclusion of Lemma~\ref{lem:froberg} holds. In particular, this implies that the degree of the polynomial $HQ_1$ is equal to the degree of the term of largest coefficient of $H$, or the least such degree if there are multiple terms with equal largest coefficient. We shall refer to this degree by $\Sigma$. Moreover, for ease of notation, we shall denote
\[ \Delta = (p-1)(d-1) + p(d-1) + (n-p)(d-2),\]
so that $\Delta$ equals the degree of $H$. 

Note that the DRL ordering with $x_n$ as the least variable is compatible with these quotients. We recall the following property that can be easily verified:

\begin{lemma}[{\cite[Lemma 1.9]{moreno}}]\label{lem:compatible}
  Let $I \in \K[x_1, \dots, x_n]$ be a polynomial ideal and let $\{g_1, \dots, g_k\}$ be a \gb of $I$ with respect to a DRL ordering with~$x_n$ as the least variable. Then $\{g_1, \dots, g_k, x_n^e\}$ is a \gb of $I + \ideal{x_n^e}$. Moreover, if $\{g_1, \dots, g_k\}$ is additionally a reduced \gb, then removing from $\{g_1, \dots, g_k, x_n^e\}$ all $g_i$ such that $x_n^e$ divides the leading term of $g_i$ gives a reduced \gb of $I + \ideal{x_n^e}$.
\end{lemma}

This compatibility can be easily seen from the corresponding staircases:
\begin{center}
\includegraphics[width=0.45\textwidth]{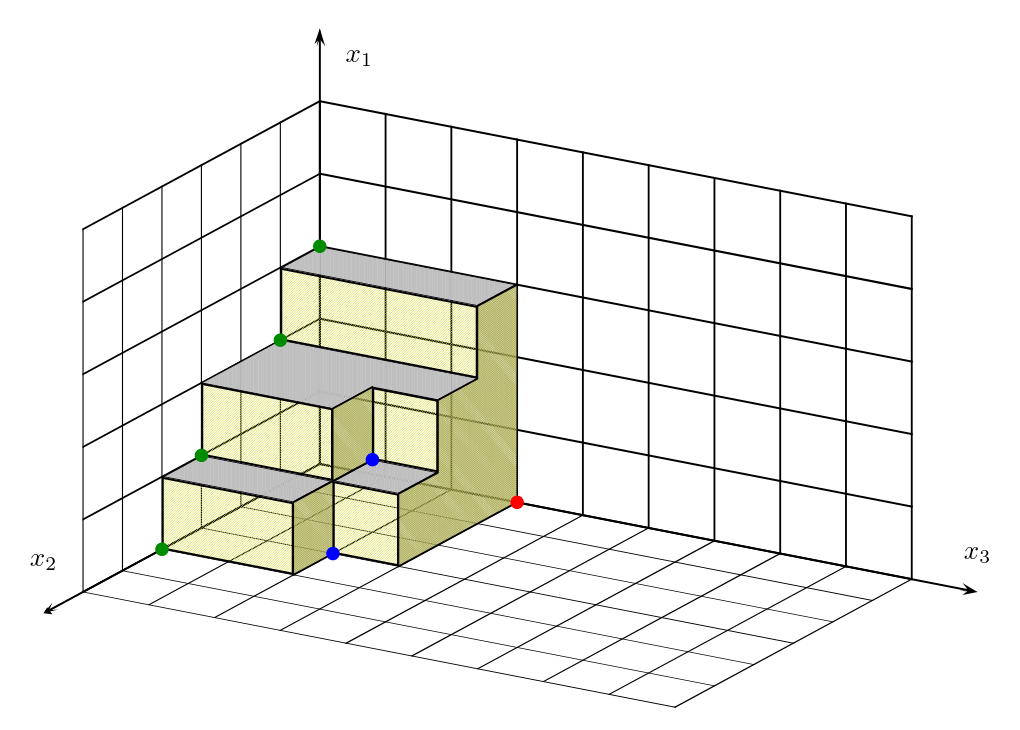}%
~%
\includegraphics[width=0.45\textwidth]{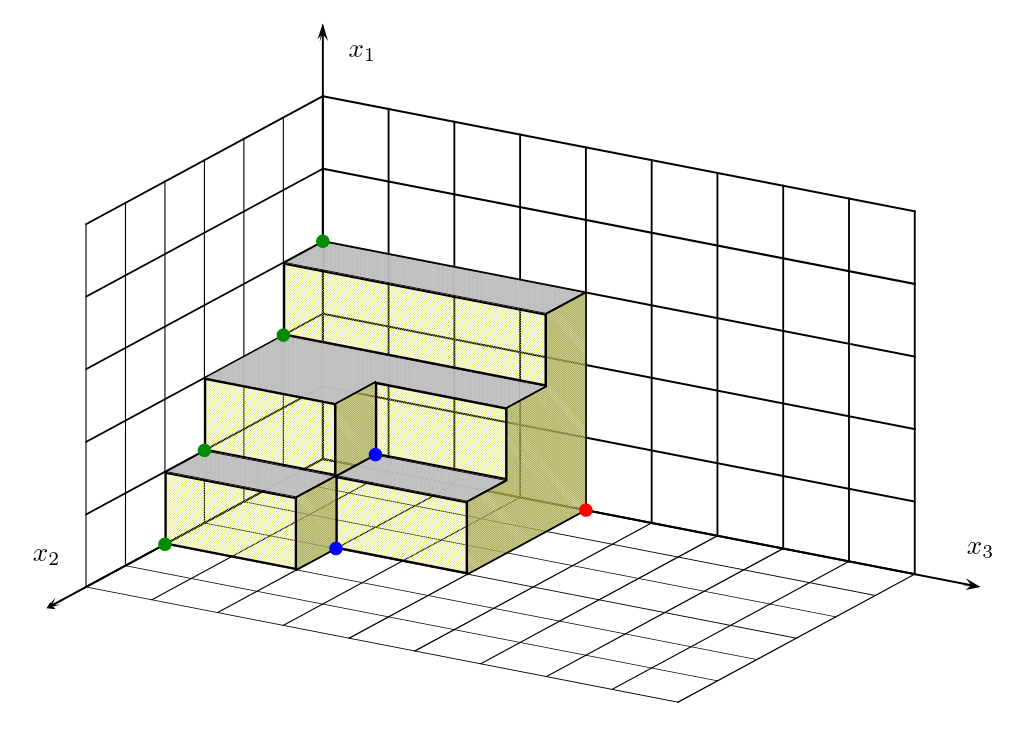}
\end{center}

Here we see the quotients by $x_3^3$ and $x_3^4$. Adding these monomials to the \gb is indicated by the red dots. As in~\cite{moreno}, for $e \geq 1$ we consider the $e$th section \[ H_e = \frac{HQ_{e+1} - HQ_e}{t^e}. \] Effectively, we consider the Hilbert series of a cross section of the DRL staircase.

\begin{center}
\includegraphics[width=0.45\textwidth]{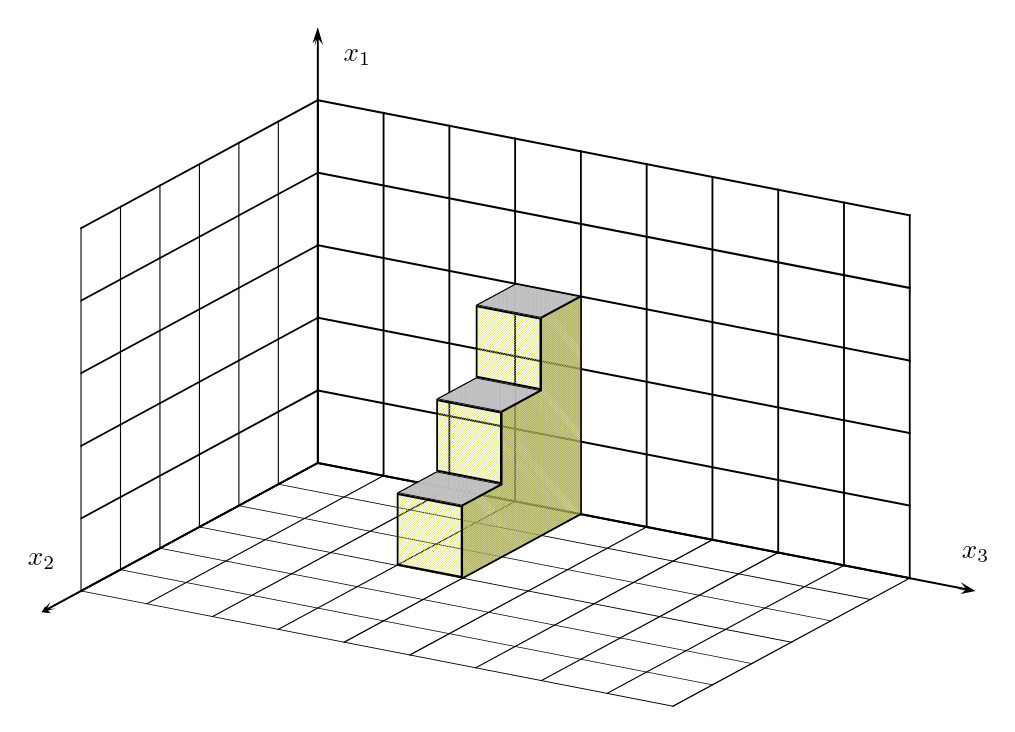}%
~%
\includegraphics[width=0.45\textwidth]{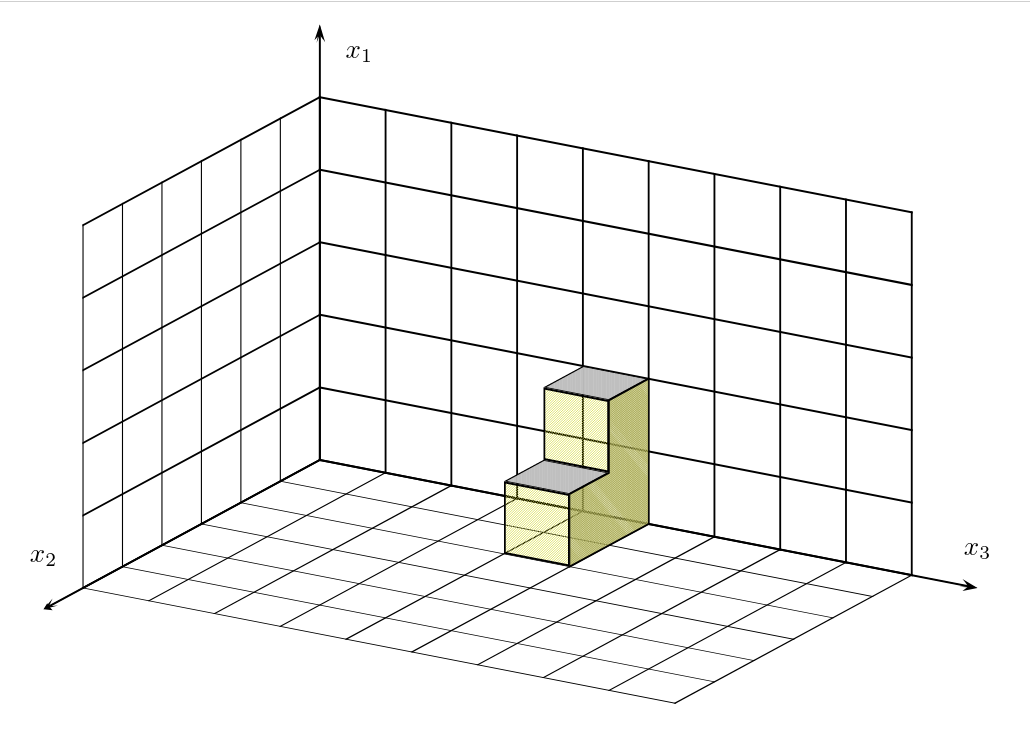}
\end{center}
Here we illustrate $t^3H_3$ and $t^4H_4$. From this example, it is clear that by scaling these polynomials, by dividing by $t^3$ and $t^4$ respectively, the difference of these polynomials tells us about how the stairs change as we increase the degree of~$x_n$. To study these sections, we first prove a result restricting the degree they can have.
\begin{lemma}\label{lem:quodeg}
  For all $e \geq 1$, $\deg HQ_{e+1} - \deg HQ_e \in \{0,1\}$.   
\end{lemma}

\begin{proof}
Let $H = \sum_{k=0}^\Delta a_k t^k$. For a given $e$, let $\sigma$ be the degree of $HQ_e$. By Lemma~\ref{lem:unimodal}, $H$ is unimodal. Therefore, by Lemma~\ref{lem:froberg},
  \[
    HQ_e = a_0 + \dots + a_{e-1} t^{e-1} + (a_e - a_0)t^e + \dots + (a_\sigma - a_{\sigma - e})t^\sigma
  \]
  Moreover, since $H$ is unimodal, the degree of $HQ_{e+1}$ is at least the degree of~$HQ_e$ and
  $\sigma \geq \Sigma$, where $\Sigma$ is the degree of $HQ_1$. For the purpose of contradiction, suppose that the degree of $HQ_{e+1}$ is $\sigma + 2$. Then
  \[  HQ_e = \left[ a_0 + \dots + (a_\sigma -  a_{\sigma - e})t^\sigma + (a_{\sigma+1} -  a_{\sigma+1-e})t^{\sigma+1} \right]_+ \]
  and 
\[HQ_{e+1} = a_0 + \dots +  (a_{\sigma + 2} -  a_{\sigma +1 - e})t^{\sigma + 2}. \] 
   This implies that
  \begin{align*}
    & a_{\sigma + 1} \leq a_{\sigma +1 - e}, \\
    & a_{\sigma + 2} > a_{\sigma + 1-e}.
  \end{align*}
  Therefore, $a_{\sigma + 1} < a_{\sigma + 2}$. This is a contradiction, as $a_\Sigma$ is the largest coefficient of $H$ and so $a_{\sigma + 1} \geq a_{\sigma + 2}$ by unimodality. Clearly, the same argument holds if the degree of $HQ_{e+1}$ is greater than $\sigma + 2$. Therefore, the degree of $HQ_{e+1}$ is either $\sigma$ or $\sigma + 1$. 
\end{proof}

 With Lemma~\ref{lem:quodeg}, we greatly restrict the possible degrees these sections can have. This allows us to prove a result on the differences of these sections. 

\begin{center}
\includegraphics[width=0.8\linewidth]{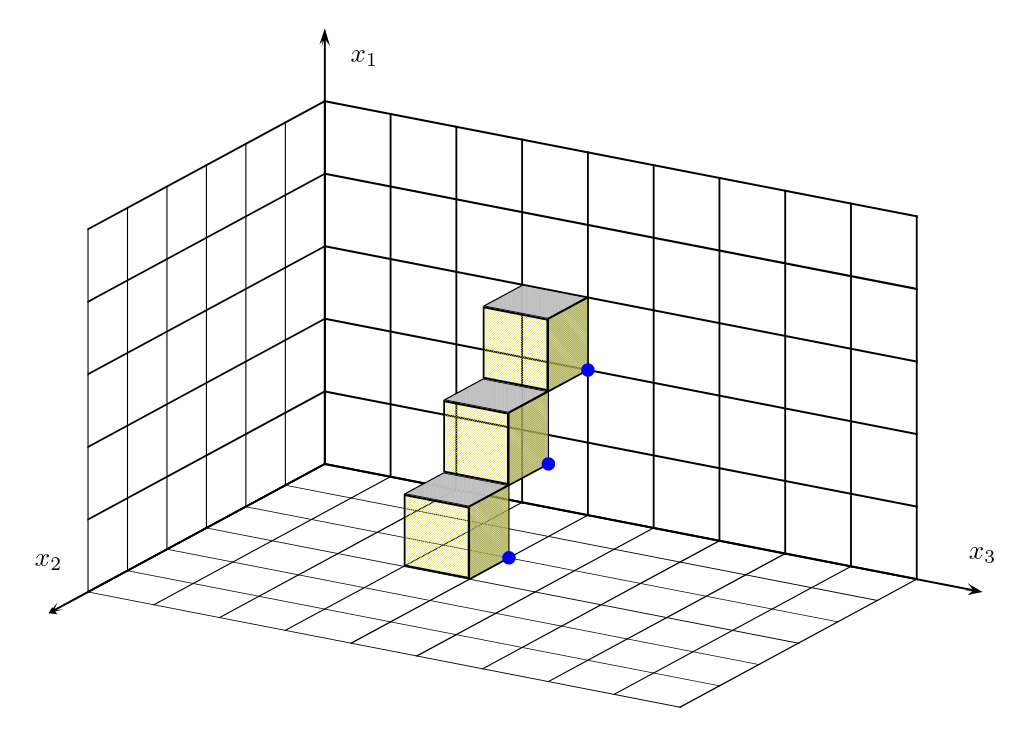}
\end{center}

 We see here that the difference of sections tells us when there are drops in the staircase as we increase the degree of $x_n$. Note that the three monomials in the illustration of the difference $t^3(H_4 - H_3)$ correspond to the three leading monomials in the reduced \gb that have degree $4$ in $x_3$. With the following lemma and proposition, we show that this correspondence always occurs.

\begin{lemma}\label{lem:diff}
  For all $e \geq 1$, the difference $H_{e+1} - H_e$ is either $0$ or a monomial.
\end{lemma}

\begin{proof}
  For a fixed $e$ we need to consider the three quotients $HQ_e$, $HQ_{e+1}$ and $HQ_{e+2}$. Let $\sigma \geq \Sigma$ 
  be the degree of $HQ_e$. Then, by Lemma~\ref{lem:quodeg}, the degree of $HQ_{e+1}$ is either $\sigma$ or $\sigma + 1$ and the degree of $HQ_{e+2}$ is between $\sigma$ and $\sigma + 2$. We consider the following four cases and show that the result holds in each:
  \begin{itemize}
  \item $\deg HQ_{e+1} - \deg HQ_e = 0$ and $\deg HQ_{e+2} - \deg HQ_{e+1} = 0$.
    Then we have the quotients:
    \begin{align*}
      HQ_e & = a_0 + \dots + a_{e-1} t^{e-1} + (a_e - a_0)t^e + \dots + (a_\sigma - a_{\sigma - e})t^\sigma,  \\
      HQ_{e+1} & = a_0 + \dots + a_{e} t^{e} + (a_{e+1} - a_0)t^{e+1} + \dots + (a_\sigma - a_{\sigma - e - 1})t^\sigma,  \\
      HQ_{e+2} & = a_0 + \dots + a_{e+1} t^{e+1} + (a_{e+2} - a_0)t^{e+2} + \dots + (a_\sigma - a_{\sigma - e - 2})t^\sigma .
    \end{align*}

    This gives the sections:
    \begin{align*}
      H_e & = a_0 + (a_1 - a_0)t + \dots + (a_{\sigma - e - 1} - a_{\sigma - e - 2})t^{\sigma - e - 1} \\ & \quad + (a_{\sigma - e} - a_{\sigma - e - 1})t^{\sigma - e}, \\
      H_{e+1} & = a_0 + (a_1 - a_0)t + \dots + (a_{\sigma - e - 1} - a_{\sigma - e - 2})t^{\sigma - e - 1}.
    \end{align*}
    Therefore, the difference is:
    \[
      H_{e+1} - H_e = (a_{\sigma - e - 1} - a_{\sigma - e})t^{\sigma - e} .
    \]

  \item $\deg HQ_{e+1} - \deg HQ_e = 1$ and $\deg HQ_{e+2} - \deg HQ_{e+1} = 0$.
    Then we have the quotients:
    \begin{align*}
      HQ_e & = a_0 + \dots + a_{e-1} t^{e-1} + (a_e - a_0)t^e + \dots + (a_\sigma - a_{\sigma - e})t^\sigma , \\
      HQ_{e+1} & = a_0 + \dots + a_{e} t^{e} + (a_{e+1} - a_0)t^{e+1} + \dots + (a_\sigma - a_{\sigma - e - 1})t^\sigma \\ & \quad +  (a_{\sigma+1} - a_{\sigma - e})t^{\sigma+1},\\
      HQ_{e+2} & = a_0 + \dots + a_{e+1} t^{e+1} + (a_{e+2} - a_0)t^{e+2} + \dots + (a_\sigma - a_{\sigma - e - 2})t^\sigma \\ & \quad + (a_{\sigma+1} - a_{\sigma - e -1})t^{\sigma+1} .
    \end{align*}
    This gives the sections:
    \begin{align*}
      H_e & = a_0 + (a_1 - a_0)t + \dots + (a_{\sigma - e} - a_{\sigma - e - 1})t^{\sigma - e} \\ & \quad +
      (a_{\sigma+1} - a_{\sigma - e})t^{\sigma - e + 1}, \\
      H_{e+1} & = a_0 + (a_1 - a_0)t + \dots  + (a_{\sigma - e} - a_{\sigma - e - 1})t^{\sigma - e} .
    \end{align*}
    Therefore, the difference is:
    \[
      H_{e+1} - H_e = (a_{\sigma - e} - a_{\sigma + 1})t^{\sigma - e + 1} .
    \]
    
  \item $\deg HQ_{e+1} - \deg HQ_e = 0$ and $\deg HQ_{e+2} - \deg HQ_{e+1} = 1$.
    Then we have  the quotients:
    \begin{align*}
      HQ_e & = a_0 + \dots + a_{e-1} t^{e-1} + (a_e - a_0)t^e + \dots + (a_\sigma - a_{\sigma - e})t^\sigma, \\
      HQ_{e+1} & = a_0 + \dots + a_{e} t^{e} + (a_{e+1} - a_0)t^{e+1} + \dots + (a_\sigma - a_{\sigma - e - 1})t^\sigma, \\
      HQ_{e+2} & = a_0 + \dots + a_{e+1} t^{e+1} + (a_{e+2} - a_0)t^{e+2} + \dots + (a_\sigma - a_{\sigma - e - 2})t^\sigma \\ & \quad + (a_{\sigma+1} - a_{\sigma - e - 1})t^{\sigma+1} .
    \end{align*}
  This gives the sections:
    \begin{align*}
      H_e & = a_0 + (a_1 - a_0)t + \dots + (a_{\sigma - e - 1} - a_{\sigma - e - 2})t^{\sigma - e - 1} \\ & \quad + (a_{\sigma - e} - a_{\sigma - e - 1})t^{\sigma - e}, \\
      H_{e+1} & = a_0 + (a_1 - a_0)t + \dots + (a_{\sigma - e - 1} - a_{\sigma - e - 2})t^{\sigma - e - 1} \\ & \quad + (a_{\sigma+1} - a_{\sigma - e - 1})t^{\sigma - e}.
    \end{align*}
    Therefore, the difference is:
    \[
      H_{e+1} - H_e = (a_{\sigma + 1} - a_{\sigma - e})t^{\sigma - e}.
    \]

  \item $\deg HQ_{e+1} - \deg HQ_e = 1$ and $\deg HQ_{e+2} - \deg HQ_{e+1} = 1$.
    Then we have the quotients:
    \begin{align*}
      HQ_e & = a_0 + \dots + a_{e-1} t^{e-1} + (a_e - a_0)t^e + \dots + (a_\sigma - a_{\sigma - e})t^\sigma, \\
      HQ_{e+1} & = a_0 + \dots + a_{e} t^{e} + (a_{e+1} - a_0)t^{e+1} + \dots + (a_\sigma - a_{\sigma - e - 1})t^\sigma \\ & \quad + (a_{\sigma+1} - a_{\sigma - e})t^{\sigma+1}, \\
      HQ_{e+2} & = a_0 + \dots + a_{e+1} t^{e+1} + (a_{e+2} - a_0)t^{e+2} + \dots \\ & \quad + (a_{\sigma+1} - a_{\sigma - e - 1})t^{\sigma+1} + (a_{\sigma+2} - a_{\sigma - e})t^{\sigma+2}.
    \end{align*}
  This gives the sections:
    \begin{align*}
      H_e & = a_0 + (a_1 - a_0)t + \dots + (a_{\sigma - e} - a_{\sigma - e - 1})t^{\sigma - e} \\ & \quad + (a_{\sigma+1} - a_{\sigma - e})t^{\sigma -e + 1}, \\
      H_{e+1} & = a_0 + (a_1 - a_0)t + \dots + (a_{\sigma - e} - a_{\sigma - e - 1})t^{\sigma - e}, \\ & \quad + (a_{\sigma+2} - a_{\sigma - e})t^{\sigma - e + 1} .
    \end{align*}
    Therefore, the difference is:
    \[
      H_{e+1} - H_e = (a_{\sigma + 2} - a_{\sigma + 1})t^{\sigma - e + 1} . \qedhere
     \]
  \end{itemize}
\end{proof}

We now can translate these results to describe the DRL staircase. For all $e \geq 0$, the sections of the staircase will be denoted by
\[
    E^e = \{x_1^{i_1}\cdots x_{n-1}^{i_{n-1}} \mid x_1^{i_1}\cdots x_{n-1}^{i_{n-1}} x_n^e \in E \}.
\]

We can now state and prove our structure result.
\begin{proposition}\label{lem:structure}
  For all  $b \in E$, either $x_n b \in E$ or $x_n b$ is a leading monomial in the reduced DRL \gb of $I$.
\end{proposition}

\begin{proof}
  Let $b \in E$ be a monomial of degree $\delta$. Assume that $x_n b \notin E$. Let $b^\prime \in E^e$ so that $b = b^\prime x_n^e$. The coefficient of the $\delta$th term of a Hilbert series is the number of monomials of degree $\delta$ under the staircase. Thus, $b$ is accounted for in the $\delta$th term of $HQ_{e+1}$. Furthermore, since $x_n^e \mid b$, $b$ is not accounted for in $HQ_{e}$ and so in the section $H_e$, $b^\prime$ is accounted for in the $(\delta - e)$th coefficient. However, since $x_n^{e+1} \nmid b$, $b$ is still accounted for in the $\delta$th term of $HQ_{e+2}$. Therefore, these parts cancel in the section $H_{e+1}$ and so in the difference $H_{e+1} - H_e$, $b^\prime$ is accounted for in \ab{the} $(\delta - e)$th term. The absolute value of the sum of the coefficients of this difference gives the number of monomials that are in $E^e$ that are not in~$E^{e+1}$. By Lemma~\ref{lem:diff}, $H_{e+1}-H_e$ is a monomial. Therefore, all monomials that are in $E^e$ and are not in $E^{e+1}$ are of the same degree and so are independent. The monomial $b^\prime$ is accounted for in the coefficient of $H_{e+1} - H_e$ and so $x_n b$ is a leading monomial in the reduced DRL \gb of $I$. \qedhere
  
\end{proof}

\Mn*

\begin{proof}
  Each column of the matrix $M_n$ is the normal form of a monomial $x_n b$ such that $b \in E$. By Lemma~\ref{lem:structure},
  either $x_n b \in E$, in which case the column is all zeroes except one entry with a value of $1$ in the row corresponding to $x_n b$, or $x_n b$ is a leading term in the reduced DRL \gb of $I$. In the latter case, the normal form is obtained from the DRL \gb without cost. Therefore, the multiplication matrix $M_n$ can be constructed for free.
\end{proof}

With this structure theorem in tow,
we aim to count the number of non-trivial columns. The following lemma gives a useful classification of this number.

\begin{lemma}\label{lem:denseH}
  If Fr\"oberg's conjecture is true, then the number of non-trivial columns of $M_n$ is equal to the largest coefficient of $H$.
\end{lemma}

\begin{proof}
  By Theorem~\ref{thm:mn}, we can count the number of non-trivial columns of $M_n$ by counting the number of polynomials in the reduced and minimal DRL \gb whose leading terms have positive degree in $x_n$. Lemma~\ref{lem:structure} implies that this number is equal to the number of monomials $b \in E$ such that $x_n b \notin E$. 
  Note that this number is also equal to the number of monomials in the section $E^0$. The monomials in this section form a monomial basis of the quotient algebra $(\K[\bx]/I)/\ideal{x_n}$. Thus, the number of non-trivial columns of $M_n$ is equal to the sum of the coefficients of the Hilbert series $HQ_1$ of this algebra. By Lemma~\ref{lem:froberg}, we can express $HQ_1$ in terms of the coefficients of $H$:
 \[
    HQ_1 = a_0 + (a_1 - a_0)t + \dots + (a_\Sigma - a_{\Sigma - 1})t^\Sigma.
  \]
  Therefore, the sum of the coefficients of $HQ_1$, and so the number of non-trivial columns of $M_n$, equals $a_\Sigma$, the largest coefficient of $H$.
\end{proof}

\subsection{Asymptotics}\label{sec:asymp}
By~\cite{sparsefglm}, the complexity of the \SFGLM algorithm depends linearly on the number of non-trivial columns of the multiplication matrix $M_n$, denoted $m$. In the previous section, we proved Lemma~\ref{lem:denseH}, meaning that we can determine this number by finding the largest coefficient of the Hilbert series $H$ from Proposition~\ref{prop:original_H}. We consider two cases. Firstly, we suppose that $d=2$. This assumption leads to a simplification of the Hilbert series so that, by Corollary~\ref{cor:new_H} and a trivial identity, it can be written as
\[
    H = \left(\sum_{k=0}^{p-1} \binom{n-p-1+k}{k} t^{k}\right) (1+t)^p.
\]
On the other hand, for any $d \geq 2$, to find an asymptotic formula for the largest coefficient of $H$ we will consider the central coefficients of polynomials of the form $(1+t+\dots+t^r)^s$ for some $r,s$. Therefore, we recall an abridged version of the following result from~\cite{star}.
\begin{proposition}[{\cite[Theorem 2]{star}}]\label{prop:star}
    Let $r, s \geq 1$ and choose $0 \leq k \leq s^{1/2}$. Then the $\rev{\frac{1}{2}}(sr+k)$th coefficient of the polynomial $(1+t+\dots+t^r)^s$ is asymptotically equal to
    \[
        \frac{1}{\sqrt{s\pi}} \sqrt {\frac{6}{r^2-1} } r^s \left(1 + O\left(\frac{k}{s}\right)\right).
    \]
\end{proposition}

We can now restate and prove our main result.
\m*
\begin{proof}
By Lemma~\ref{lem:denseH}, $m$ is equal to the largest coefficient of the Hilbert series~$H$.
First, assume that $d=2$. Then the Hilbert series can be written as 
\[
    H = \left(\sum_{k=0}^{p-1} \binom{n-p-1+k}{k} t^{k}\right) (1+t)^p = \sum_{k=0}^{p(p-1)} h_k t^k.
\]
In this setting, we consider the binomial coefficients:
 \[
   (1+t)^p = \sum_{k=0}^p \binom{p}{k} t^k = \sum_{k=0}^p a_k t^k.
 \]
 We shall prove our first result by finding the degree of the term of $H$ with the largest coefficient. The number $m$ can then be found by a convolution formula.
 
 Firstly, note that $(1+t)^p$ is a symmetric unimodal polynomial. Therefore, its largest coefficient is at the term of degree $\floor{\frac{p}{2}}$. Since this polynomial is unimodal, 
 \[
 h_{\floor{\frac{3p}{2}}} = \sum_{k=0}^{p-1} \binom{n-2-k}{p-1-k} a_{\floor{\frac{p}{2}} + 1 + k} \leq \sum_{k=0}^{p-1} \binom{n-2-k}{p-1-k} a_{\floor{\frac{p}{2}}+k} = h_{\floor{\frac{3p}{2}}-1}.
 \]
 By Lemma~\ref{lem:unimodal}, $H$ is unimodal and so the largest coefficient of $H$ is at least $h_{\floor{\frac{3p}{2}}-1}$. We now show that the previous coefficient of $H$ is also no more than $h_{\floor{\frac{3p}{2}}-1}$. By unimodality, this shows that $h_{\floor{\frac{3p}{2}}-1}$ is the largest coefficient. Hence,
 \[
   h_{\floor{\frac{3p}{2}}-1} = \sum_{k=0}^{p-1} \binom{n-p-1+k}{k} \binom{p}{\floor{\frac{3p}{2}} - 1 - k}
 \]
 and
 \[
   h_{\floor{\frac{3p}{2}}-2} = \sum_{k=0}^{p-1} \binom{n-p-1+k}{k} \binom{p}{\floor{\frac{3p}{2}} - 2 - k}.
 \]
 As $n \to \infty$ we can write this as:
 \[
   h_{\floor{\frac{3p}{2}}-1} = \binom{n-2}{p-1} \binom{p}{\floor{\frac{p}{2}}} + O(n^{p-2})
 \]
 and
 \[
   h_{\floor{\frac{3p}{2}}-2} = \binom{n-2}{p-1} \binom{p}{\floor{\frac{p}{2}}-1} + O(n^{p-2}).
 \]
 Therefore, \[
   h_{\floor{\frac{3p}{2}}-1} -  h_{\floor{\frac{3p}{2}}-2} = \binom{n-2}{p-1} \left(\binom{p}{\floor{\frac{p}{2}}} -  \binom{p}{\floor{\frac{p}{2}}-1}\right) + O(n^{p-2})
 \]
 If $p=1$, then $H = 1+t$, and so the largest coefficient is indeed $h_0 = 1$. Otherwise, $\binom{p}{\floor{\frac{p}{2}}} > \binom{p}{\floor{\frac{p}{2}}-1}$ and so this difference tends to positive infinity as $n \to \infty$. 
 
 Therefore, for sufficiently large $n$, the largest coefficient is $h_{\floor{\frac{3p}{2}}-1}$.

Suppose now that $d>2$. We return to the Hilbert series form given in Corollary~\ref{cor:new_H} along with a trivial identity:
\[
H = \left(\sum_{k=0}^{p-1} \binom{n-p-1+k}{k} t^{k(d-1)}\right)(1+t+\dots+t^{d-1})^p(1+t+\dots+t^{d-2})^{n-p}.
\]
Firstly, consider the binomial sum factor. Note that as $n\to\infty$, the dominant term is the term of highest degree. Specifically, we may write
        \[\sum_{k=0}^{p-1} \binom{n-p-1+k}{k}t^{k(d-1)} = \binom{n-2}{p-1}t^{(p-1)(d-1)} + O(n^{p-2} t^{(p-2)(d-1)}). \]
    Therefore, since we only consider the largest coefficient of $H$ as $n \to \infty$, we see that this is equal to the largest coefficient of the polynomial
    \[h = \binom{n-2}{p-1}(1+t+\cdots+t^{d-1})^p(1+t+\cdots+t^{d-2})^{n-p}.\]
    Thus, we can replace the binomial sum in the expression we consider with just a binomial coefficient. 
    
    For ease of notation, denote the other factors of $h$ by $f_1 = (1+ t + \cdots + t^{d-1})^p$ and $f_2 = (1+t+\cdots+t^{d-2})^{n-p}$. By~\cite[Proposition 2.2]{moreno}, these polynomials are symmetric. In particular, this means that $a_i = a_{(d-2)(n-p) - i}$, where \[ f_2 = \sum_{i=0}^{(d-2)(n-p)} a_i t^i.\] Then, by Lemma~\ref{lem:strong_unimodal_ex} the polynomial $f_2$ is unimodal and so its largest coefficient is the central one. Therefore, by Proposition~\ref{prop:star}, the largest coefficient of $f_2$ is asymptotically equal to
    \[
        \frac{1}{\sqrt{(n-p)\pi}} \sqrt { \frac{6}{(d-1)^2-1} }  (d-1)^{n-p}.
    \]
    Also by Proposition~\ref{prop:star}, since $p(d-1) + 1$ is fixed as $n \to \infty$, the central $p(d-1) + 1$ coefficients of $f_2$ tend to its largest coefficient. Note that for sufficiently large $n$, the largest coefficient of the product $f_1 f_2$ depends only on the central $p(d-1) +1$ coefficients of $f_2$, since $f_1$ does not depend on $n$.
    Therefore, since the sum of the coefficients of $f_1$ equals $d^p$, as $n \to \infty$, the largest coefficient of $H$ is asymptotically equal to
    \[
        \frac{1}{\sqrt{(n-p)\pi}} \sqrt { \frac{6}{(d-1)^2-1} } d^p (d-1)^{n-p}\binom{n-2}{p-1}. 
    \]
    We conclude that, for $d \geq 3$ and $n \to \infty$, the number of non-trivial columns of~$M_n$ is asymptotically equal to 
    \[ 
        m \approx \frac{1}{\sqrt{(n-p)\pi}}\sqrt { \frac{6}{(d-1)^2-1} } d^p (d-1)^{n-p} \binom{n-2}{p-1}. \qedhere
    \]
\end{proof}

\comp*
\begin{proof}

Firstly, by Definition~\ref{def:gdi}, we may apply the shape position variant of the \SFGLM algorithm. Assuming the multiplication matrix $M_n$ is constructed, its complexity is $O(m D^2+n D\log^2(D))$, where $m$ is the number of non-trivial columns of the multiplication matrix $M_n$ and $D$ is the degree of the ideal $I$~\cite[Theorem 3.2]{sparsefglm}. By Theorem~\ref{thm:mn}, the construction of the matrix $M_n$ requires no arithmetic operations. Recall that the degree of the ideal $I$ \ab{is equal to}
\[ D = d^p (d-1)^{n-p} \binom{n-1}{p-1}.\] Then, for $d \geq 3$, by Theorem~\ref{thm:m}, as $n \to \infty$,
    \[ 
        m \approx \frac{1}{\sqrt{(n-p)\pi}}\sqrt { \frac{6}{(d-1)^2-1} } d^p (d-1)^{n-p} \binom{n-2}{p-1}. 
    \]
    Since the dominant term of the complexity is $O(mD^2)$, substituting the formula for $D$ and the asymptotics of $m$ gives the complexity result.
    
    The complexity gain is then \[ O\left(\frac{m D^2}{n D^3}\right) = O\left(\frac{m}{n D}\right) \approx O\left( \frac{\sqrt{n-p}}{n^2(d-1)}\right).\qedhere \]
\end{proof}

\section{Experiments}\label{sec:exp}
In this section, we test \ab{the practical accuracy of} our formulae in Theorem~\ref{thm:m}, for the number of dense columns of the multiplication matrix $M_n$. For $d=2$ we use our exact formula~\ab{\eqref{exact:d=2}}, while for $d \geq 3$ we use the asymptotic formula~\ab{\eqref{asympt:d>2}}. The matrix density refers to the number of non-zero entries of $M_n$ divided by its total number of entries. As seen in Theorem~\ref{thm:comp}, the matrix density gives an idea of the complexity gain of using \SFGLM over \FGLM for the change of ordering.

\begin{table}[htb]
    \centering
    \begin{tabular}{ccSSS}
    \hline
    {Parameters} & {Degree} & \multicolumn{3}{S}{Matrix~Density}\\
    \cline{3-5} 
    $(d, p, n)$ & $D$ & Actual & {Theoretical} & Asymptotic \\
    \hline
    $(2, 4, 9)$ 
    & 896 & 30.17\% & \color{blue}{30.80\%} & \color{blue}{30.80\%} \\
    $(2, 4, 10)$ 
    & 1344 & 31.13\% & \color{blue}{31.77\%} & \color{blue}{31.77\%} \\ 
    $(2, 4, 11)$ 
    & 1920 & 31.86\% & \color{blue}{32.50\%}& \color{blue}{32.50\%} \\

    $(3, 3, 6)$ 
    & 2160 & 17.52\% & \color{blue}{18.52\%} & \color{blue}{27.73\%}\\
    $(3, 3, 7)$ 
    & 6480 & 17.39\% & \color{blue}{18.31\%} & \color{blue}{26.62\%}\\
    $(3, 3, 8)$ 
    & 18144 & 17.63\% & \color{blue}{18.72\%} & \color{blue}{25.50\%}\\
    
    $(4, 2, 5)$ 
    & 1728 & 14.46\% & \color{blue}{15.45\%} & \color{blue}{21.24\%}\\
    $(4, 2, 6)$ 
    & 6480 & 14.11\% & \color{blue}{15.13\%} & \color{blue}{19.56\%}\\
    
     $(5, 2, 5)$ 
    & 6400 & 11.00\%& \color{blue}{11.94\%} & \color{blue}{15.47\%}\\
     $(6, 2, 5)$ 
    & 18000 & 8.80\%& \color{blue}{\, 9.63\%} & \color{blue}{12.22\%}\\
    \hline
\end{tabular}
\caption{Density of multiplication matrix $M_n$ for generic critical point systems}
\label{tab:crit}
\end{table}

Table~\ref{tab:crit} originates as a cropped version of~\cite[Table 2]{sparsefglm}. There, the authors give the values in the ``Actual'' column, obtained by computing the multiplication matrix and calculating exactly the number of non-zero entries, but the entries in the theoretical and asymptotic columns were blank. Now, with Theorem~\ref{thm:m} we can complete this table, and we put the new entries in blue. The entries of the theoretical and asymptotic columns are the values of $m/D$, approximately the density of non-zero entries, for the varying parameters. In the theoretical column, the value of $m$ is taken to be the largest coefficient of the Hilbert series. Then for the asymptotic column we take $m$ as in Theorem~\ref{thm:m}.

Exceptionally, in Figure~\ref{fig:graphs}, we consider the generic determinantal ideals defined by two quartics, and also the generic determinantal ideals defined by four polynomials of degree $8$, with an increasing number of variables $n$. 

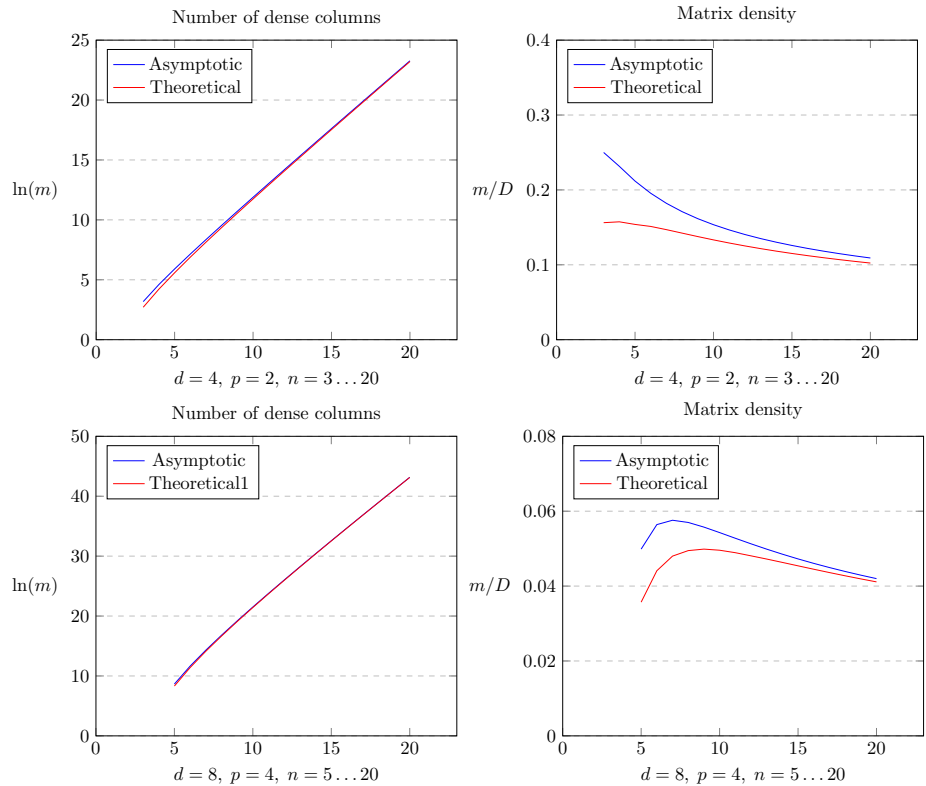
\begin{figure}[ht]
    \centering
    \begin{center}
    \begin{tikzpicture}[scale = 0.7]
    \begin{axis}[
        title={Number of dense columns},
        xlabel={$d=4, \; p=2, \; n=3\dots20$},
        ylabel={ln($m$)},
        ylabel style={rotate=-90},
        xmin=0, xmax=23,
        ymin=0, ymax=25,
        xtick={0,5,10,15,20},
        ytick={0,5,10,15,20,25},
        legend pos=north west,
        ymajorgrids=true,
        grid style=dashed,
    ]
    
    \addplot[
        color=blue,
        mark=cross,
        ]
        coordinates {(3, 3.178053830)(4, 4.605170186)(5, 5.902633333)(6, 7.144407180)(7, 8.354203563)(8, 9.543951376)(9, 10.71962640)(10, 11.88500613)(11, 13.04250599)(12, 14.19379835)(13, 15.34006574)(14, 16.48218354)(15, 17.62081720)(16, 18.75648345)(17, 19.88959217)(18, 21.02047372)(19, 22.14939832)(20, 23.27658982)
        };
        \addplot[
        color=red,
        mark=cross,
        ]
        coordinates {(3, 2.708050201)(4, 4.219507705)(5, 5.583496309)(6, 6.887552572)(7, 8.139732280)(8, 9.360396968)(9, 10.55952978)(10, 11.74331470)(11, 12.91560731)(12, 14.07902570)(13, 15.23539950)(14, 16.38606013)(15, 17.53200175)(16, 18.67398539)(17, 19.81260530)(18, 20.94833373)(19, 22.08155177)(20, 23.21257104)
        };
        \legend{Asymptotic, Theoretical  }
    \end{axis}
\end{tikzpicture}%
~%
\begin{tikzpicture}[scale = 0.7]
    \begin{axis}[
        scaled ticks=false,
        title={Matrix density},
        xlabel={$d=4, \; p=2, \; n=3\dots20$},
        ylabel={$m/D$},
        ylabel style={rotate=-90},
        xmin=0, xmax=23,
        ymin=0, ymax=0.4,
        xtick={0,5,10,15,20},
        ytick={0,0.1,0.2,0.3,0.4},
        legend pos=north west,
        ymajorgrids=true,
        grid style=dashed,
    ]
    
    \addplot[
        color=blue,
        mark=cross,
        ]
        coordinates {(3, 0.2500000000)(4, 0.2314814815)(5, 0.2118055556)(6, 0.1955246914)(7, 0.1820987654)(8, 0.1709778562)(9, 0.1615905064)(10, 0.1535536165)(11, 0.1465808058)(12, 0.1404633709)(13, 0.1350426247)(14, 0.1301975974)(15, 0.1258343891)(16, 0.1218788799)(17, 0.1182718370)(18, 0.1149652970)(19, 0.1119199870)(20, 0.1091034156)
        };
        \addplot[
        color=red,
        mark=cross,
        ]
        coordinates {(3, 0.1562500000)(4, 0.1574074074)(5, 0.1539351852)(6, 0.1512345679)(7, 0.1469478738)(8, 0.1423059965)(9, 0.1376850423)(10, 0.1332674982)(11, 0.1291117335)(12, 0.1252327612)(13, 0.1216227730)(14, 0.1182652358)(15, 0.1151402816)(16, 0.1122276848)(17, 0.1095081311)(18, 0.1069637858)(19, 0.1045784659)(20, 0.1023376275)};
        \legend{Asymptotic, Theoretical  }
    \end{axis}
\end{tikzpicture}

\begin{tikzpicture}[scale = 0.7]
    \begin{axis}[
        title={Number of dense columns},
        xlabel={$d=8, \; p=4, \; n=5\dots20$},
        ylabel={ln($m$)},
        ylabel style={rotate=-90},
        xmin=0, xmax=23,
        ymin=0, ymax=50,
        xtick={0,5,10,15,20},
        ytick={0,10,20,30,40,50},
        legend pos=north west,
        ymajorgrids=true,
        grid style=dashed,
    ]
    
    \addplot[
        color=blue,
        mark=cross,
        ]
        coordinates {(5, 8.651724084)(6, 11.63722942)(7, 14.29669025)(8, 16.79190618)(9, 19.18586031)(10, 21.51061330)(11, 23.78491322)(12, 26.02073262)(13, 28.22620498)(14, 30.40711694)(15, 32.56773627)(16, 34.71130278)(17, 36.84033513)(18, 38.95683066)(19, 41.06240039)(20, 43.15836283)
        };
        \addplot[
        color=red,
        mark=cross,
        ]
        coordinates {(5, 8.317521996)(6, 11.39073523)(7, 14.11480000)(8, 16.65043679)(9, 19.07442470)(10, 21.41982137)(11, 23.70905372)(12, 25.95608722)(13, 28.17170215)(14, 30.36074058)(15, 32.52786564)(16, 34.67671943)(17, 36.81010474)(18, 38.93022463)(19, 41.03884214)(20, 43.13739021)
        };
        \legend{Asymptotic, Theoretical1  }
    \end{axis}
\end{tikzpicture}%
~%
\begin{tikzpicture}[scale = 0.7]
    \begin{axis}[
        scaled ticks=false,
        title={Matrix density},
        xlabel={$d=8, \; p=4, \; n=5\dots20$},
        ylabel={$m/D$},
        ylabel style={rotate=-90},
        xmin=0, xmax=23,
        ymin=0, ymax=0.08,
        xtick={0,5,10,15,20},
        ytick={0,0.02,0.04,0.06,0.08},
        legend pos=north west,
        ymajorgrids=true,
        grid style=dashed,
    ]
    
    \addplot[
        color=blue,
        mark=cross,
        ]
        coordinates {(5, 0.04987444196)(6, 0.05641940370)(7, 0.05758238145)(8, 0.05699175636)(9, 0.05575387884)(10, 0.05428916799)(11, 0.05277510307)(12, 0.05128996214)(13, 0.04986778505)(14, 0.04852177927)(15, 0.04725506632)(16, 0.04606588660)(17, 0.04495021400)(18, 0.04390309795)(19, 0.04291935578)(20, 0.04199392425)
        };
        \addplot[
        color=red,
        mark=cross,
        ]
        coordinates {(5, 0.03570556641)(6, 0.04409378986)(7, 0.04800601881)(8, 0.04947350341)(9, 0.04987457633)(10, 0.04957728601)(11, 0.04891969350)(12, 0.04807920082)(13, 0.04722258967)(14, 0.04632289788)(15, 0.04540804292)(16, 0.04450000628)(17, 0.04361168539)(18, 0.04275041292)(19, 0.04192006811)(20, 0.04112237278)
        };
        \legend{Asymptotic, Theoretical  }
    \end{axis}
\end{tikzpicture}
\end{center}
    \caption{Comparison of our asymptotic formulae against the theoretical number of dense columns and matrix density for generic critical point systems with parameters $(d,p,n)$.}
    \label{fig:graphs}
\end{figure}

Note that the number of dense columns increases exponentially with $n$ in about the same exponent for either the theoretical or the asymptotic in both examples. On the other hand, the matrix density can have different behaviours as the number of variables $n$ increases for different degrees $d$ and number of polynomials $p$. However, in both examples we see that the asymptotic approximation of the matrix density is rather inaccurate for small~$n$. But, for moderate $n$, the approximation becomes good.

\paragraph*{Acknowledgements} The authors are supported by the ANR grants
ANR-18-CE33-0011 \textsc{Sesame}, ANR-19-CE40-0018 \textsc{De Rerum Natura} and
ANR-19-CE48-0015 \textsc{ECARP}, the PGMO grant \textsc{CAMiSAdo} and the
European Union's Horizon 2020 research and innovation programme under the Marie
Sklodowska-Curie grant agreement N. 813211 (POEMA). \rev{We would also like to 
thank the referees for their careful reading and very helpful comments.}

\bibliographystyle{elsarticle-harv}
\bibliography{lemma}
\end{document}